\documentclass[11pt,dvips,twoside]{article}

\usepackage{pslatex}
\usepackage{fancyhdr}
\usepackage{graphicx}
\usepackage{geometry}

\RequirePackage{amsfonts,amssymb,amsmath,amscd,amsthm}
\RequirePackage{txfonts}
\RequirePackage{graphicx}
\RequirePackage{xcolor}
\RequirePackage{geometry}
\RequirePackage{enumerate}

\def\figurename{Figure} 
\makeatletter
\renewcommand{\fnum@figure}[1]{\figurename~\thefigure.}
\makeatother

\def\tablename{Table} 
\makeatletter
\renewcommand{\fnum@table}[1]{\tablename~\thetable.}
\makeatother

\usepackage{amsmath}
\usepackage{amssymb}
\usepackage{amsfonts}
\usepackage{amsthm,amscd}

\newtheorem{theorem}{Theorem}[section]
\newtheorem{lemma}[theorem]{Lemma}
\newtheorem{corollary}[theorem]{Corollary}
\newtheorem{proposition}[theorem]{Proposition}
\theoremstyle{definition}
\newtheorem{definition}[theorem]{Definition}
\newtheorem{example}[theorem]{Example}

\theoremstyle{remark}
\newtheorem{remark}[theorem]{Remark}

\numberwithin{equation}{section}

\begin{document}

\title{\bfseries\scshape{Ternary Leibniz color algebras and beyond}}

\author{\bfseries\scshape Ibrahima BAKAYOKO\thanks{e-mail address: ibrahimabakayoko27@gmail.com}\\
D\'epartement de Math\'ematiques,\\
Centre Universitaire de N'Z\'er\'ekor\'e/UJNK\\
BP 50 N'Z\'er\'ekor\'e, Guin\'ee.
}
 
\date{}
\maketitle 


\noindent\hrulefill

\noindent {\bf Abstract.} 
The purpose of this paper is to generalize some results on ternary Leibniz algebras to the case of ternary Leibniz color algebras.
 In particular, we study color Lie triple systems.
In order to produce examples of ternary Leibniz
color algebras from Leibniz color algebras, several results on Leibniz color algebras are given.
 Then we  introduce and give some constructions of ternary Leibniz-Nambu-Poisson color algebras.
 The relationship between associative trialgebras and $-1$-tridendriform algebras are investigated.
 Moreover, we give some methods of constructing modules over  ternary Leibniz-Nambu-Poisson color algebras.

\noindent \hrulefill

\vspace{.3in}


\vspace{.08in} \noindent \textbf{Keywords}: Ternary Leibniz color algebras, color Jordan triple systems, color Lie triple systems, 
ternary Leibniz-Nambu-Poisson color algebras, associative color trialgebras and $q$-tridendriform color algebras. 
\vspace{.3in}
\vspace{.2in}
\section{Introduction}
The notion of n-Lie algebras were introduced by Filippov \cite{F} in 1985  as a natural generalization of Lie algebras. 
More precisely, $n$-Lie algebras are vector spaces $V$ equipped with $n$-ary operation which is skew symmetric 
for any pair of variables and satisfies the following identity :
\begin{eqnarray}
 [[x_1, x_2, \dots, x_n], y_1, y_2, \dots, y_{n-2}, y_{n-1}]=\sum_{i=1}^n[x_1, x_2, \dots, x_{i-1}, [x_i, y_1, y_2, \dots, y_{n-2}, y_{n-1}], 
x_{i+1}, \dots, x_n].\label{F}
\end{eqnarray}
For $n=3$, it reads
\begin{eqnarray}
 [[x, y, z], t, u]=[x, y, [z, t, u]]+[x, [y, t, u], z]+[[x, t, u], y, z].\nonumber
\end{eqnarray}
Whenever the identity (\ref{F}) is satisfied and the bracket fails to be totally skew symmetric we  obtain n-Leibniz algebras  \cite{JM}.
Moreover, when the bracket is skew-symmetric with respect to the last two variables, $(V, [-, -, -])$ is said to quasi-Lie 3-algebras \cite{JMC}.
If in addition, 
\begin{eqnarray}
[x, y, z]+[y, z, x]+[z, x, y]=0\nonumber
\end{eqnarray}
is satisfied for any $x, y, z\in V$, $(V, [-, -, -])$ is called a Lie triple system \cite{JMC}.

The n-Lie algebras found their applications in many fields of mathematics and Physics. For instance, Takhtajan has developed the foundations of 
the theory of Nambu-Poisson manifolds \cite{TL}.
The general cohomology theory for n-Lie algebras and Leibniz n-algebras was established in \cite{RM}.
The structure and classification theory of finite dimensional n-Lie algebras was given by Ling \cite{LW}  and many
other authors. For more details of the theory and applications of n-Lie algebras, see \cite{DI}
and the references therein.
 
Specially, 3-Lie algebras are applied to the study of the gauge symmetry and supersymmetry of multiple coincident M2-branes in \cite{BL}.
The authors in \cite{MA} studied non-commutative ternary Nambu-Poisson algebras and their Hom-type version. They provided construction results 
dealing with tensor product and direct sums of two (non- commutative) ternary (Hom-) Nambu-Poisson algebras, and examples and a 3-dimensional
 classification of non-commutative ternary Nambu-Poisson algebras.

The concept of n-Lie algebras are extended to the graded case by  Zhang T. in \cite{TZ}, in which he studied the cohomology and deformations of
 n-Lie colour algebras when $n=3$, as well as the abelian extensions of 3-Lie colour algebras. 
Ivan Kaygorodova and  Yury Popovc studied generalized derivations of n-ary color algebras \cite{IY}.

The aim of this paper is to construct ternary Leibniz-Nambu-Poisson color algebras. The paper is organized as follows. In section two, we recall basic 
notions concerning associative color algebras, Lie color algebras and averaging operator. Construction and various examples of
 Leibniz color algebras are given. In section three, we introduce ternary Leibniz color algebras and  associative color trialgebras. 
We give some constructions of ternary Leibniz color algebras (and color Lie triple systems) from other algebraic structures.
 Then we study the relationship between associative trialgebras and $-1$-tridendriform algebras.
 In section four, we give some constructions of ternary Leibniz-Nambu-Poisson color algebras and
give two constructions methods of modules over ternary Leibniz-Nambu-Poisson color algebras.

Throughout this paper, all graded vector spaces are assumed to be over a field $\mathbb{K}$ of characteristic different from 2.

\pagestyle{fancy} \fancyhead{} \fancyhead[EC]{Ibrahima Bakayoko} 
\fancyhead[EL,OR]{\thepage} \fancyhead[OC]{Ternary Leibniz color algebras and beyond} \fancyfoot{}
\renewcommand\headrulewidth{0.5pt}

\section{Preliminaries}
In this section, we give the definitions of associative color algebras, Lie color  algebras,
 averaging operators on associative and Lie color algebras, and constructions of  Leibniz color algebras.
\subsection{Basic notions}
\begin{definition}
 \begin{enumerate}
  \item [1)] Let $G$ be an abelian group. A vector space $V$ is said to be a $G$-graded if, there exists a family $(V_a)_{a\in G}$ of vector 
subspaces of $V$ such that
$$V=\bigoplus_{a\in G} V_a.$$
\item [2)] An element $x\in V$ is said to be homogeneous of degree $a\in G$ if $x\in V_a$. We denote $\mathcal{H}(V)$ the set of all homogeneous elements
in $V$.
\item [3)] Let $V=\oplus_{a\in G} V_a$ and $V'=\oplus_{a\in G} V'_a$ be two $G$-graded vector spaces. A linear mapping $f : V\rightarrow V'$ is said 
to be homogeneous of degree $b$ if 
$$f(V_a)\subseteq  V'_{a+b}, \forall a\in G.$$
If, $f$ is homogeneous of degree zero i.e. $f(V_a)\subseteq V'_{a}$ holds for any $a\in G$, then $f$ is said to be even.
 \end{enumerate}
\end{definition}
\begin{definition}
  \begin{enumerate}
   \item [1)] An algebra $(A, \cdot)$ is said to be $G$-graded if its underlying vector space is $G$-graded i.e. $A=\bigoplus_{a\in G}A_a$, and if furthermore 
$A_a\cdot A_b\subseteq A_{a+b}$, for all $a, b\in G$.

 Let $A'$ be another $G$-graded algebra.
\item  [2)] A morphism $f : A\rightarrow A'$ 
of $G$-graded algebras
is by definition an algebra morphism from $A$ to $A'$ which is, in addition an even mapping.
  \end{enumerate}
\end{definition}

\begin{definition}
 Let $G$ be an abelian group. A map $\varepsilon :G\times G\rightarrow {\bf \mathbb{K}^*}$ is called a skew-symmetric bicharacter on $G$ if the following
identities hold, 
\begin{enumerate} 
 \item [(i)] $\varepsilon(a, b)\varepsilon(b, a)=1$,
\item [(ii)] $\varepsilon(a, b+c)=\varepsilon(a, b)\varepsilon(a, c)$,
\item [(iii)]$\varepsilon(a+b, c)=\varepsilon(a, c)\varepsilon(b, c)$,
\end{enumerate}
$a, b, c\in G$,
\end{definition}

If x and y are two homogeneous elements of degree $a$ and $b$ respectively and $\varepsilon$ is a skew-symmetric bicharacter, 
then we shorten the notation by writing $\varepsilon(x, y)$ instead of $\varepsilon(a, b)$.

Let us also recall these definitions.
\begin{definition}
An associative color algebra is a $G$-graded algebra $(A, \cdot)$ together with a bicharacter 
$\varepsilon :G\times G\rightarrow {\bf \mathbb{K}^*}$ such that 
\begin{eqnarray}
(x\cdot y)\cdot z &=&x\cdot(y\cdot z) \qquad\qquad(\mbox{associativity})
 \end{eqnarray}
for all $x, y, z\in\mathcal{H}(A)$.
\end{definition}

\begin{definition}
 A Lie color  algebra is a triple  $(A, [\cdot, \cdot], \varepsilon)$ in which $(A, [\cdot, \cdot])$ is a $G$-graded algebra and  
$\varepsilon :G\times G\rightarrow {\bf \mathbb{K}^*}$ is a bicharacter such that
\begin{eqnarray}
 [x, y]=-\varepsilon(x, y)[y, x],\qquad\qquad\qquad\qquad (\varepsilon\mbox{-skew-symmetry})\label{ss}\qquad\\
\varepsilon(z, x)[x, [y, z]]+\varepsilon(x, y)[y, [z, x]]+\varepsilon(y, z)[z, [x, y]]=0,\;
 (\varepsilon\mbox{-Jacobi identity})\quad \label{chli}
\end{eqnarray}
for any $x, y, z\in\mathcal{H}(A)$.
\end{definition}

\begin{definition}\label{bk3}
1) An averaging operator over an associative color algebra  $(A, \cdot, \varepsilon)$ is an even linear map $\alpha : A\rightarrow A$ such
that 
$$\alpha(\alpha(x)\cdot y)=\alpha(x)\cdot \alpha(y)=\alpha(x\cdot \alpha(y)),$$
for all $x, y\in\mathcal{H}(A)$.\\
2) An averaging operator over a Lie color algebra  $(A, [\cdot, \cdot], \varepsilon)$ is an even linear map $\alpha : A\rightarrow A$ such
that 
$$\alpha([\alpha(x), y])=[\alpha(x), \alpha(y)]=\alpha([x, \alpha(y)]),$$
for all $x, y\in\mathcal{H}(A)$.
\end{definition}
\begin{definition}
 Let $(A, \cdot)$ be a $G$-graded algebra and $\lambda \in \mathbb{K}$. An even linear map $R : A\times A\rightarrow A$
 is called a Rota-Baxter operator of weight $\lambda$ on $A$ if it satisfies the identity
\begin{eqnarray}
R(x)\cdot R(y) = R\Big(R(x)\cdot y + x\cdot R(y) +\lambda x\cdot y\Big),\quad (\mbox{\it Rota-Baxter identity}) \label{rbi}
\end{eqnarray}
for any $x, y\in\mathcal{H}(A)$.
\end{definition}

\subsection{Constructing Leibniz color algebras}
In this subsection, we give some results on Leibniz color algebras. They will provide examples of
 ternary Leibniz color algebras from other algebraic structures.
\begin{definition}\cite{BD} 
 A Leibniz color algebra is a $G$-graded vector space $L$ together with an even bilinear map 
$[-, -] : L\otimes L\rightarrow L$ and a bicharacter $\varepsilon : G\otimes G\rightarrow \mathbb{K}^*$  such that
 \begin{eqnarray}
   [[x, y], z]=[x, [y, z]]+\varepsilon(y, z)[[x, z], y] \label{cpa}
 \end{eqnarray}
holds, for all $x, y, z\in \mathcal{H}(L)$.
\end{definition}
Observe that a Lie color algebra is a Leibniz color algebra with $\varepsilon$-skew-symmetric bracket.
\begin{example}\label{a3c}
 Let $(A, \cdot, \varepsilon)$ be an associative color algebra. Then $(A, [-, -], \varepsilon)$ is a Lie color algebra with the bracket
\begin{eqnarray}
 [x, y]:=x\cdot y-\varepsilon(x, y)y\cdot x,\nonumber
\end{eqnarray}
for any $x, y\in \mathcal{H}(A)$.
\end{example}
\begin{example}\label{bk1}
 Let $(A, [-, -], \varepsilon)$ be a Lie color algebra and $\alpha : A\rightarrow A$ be an averaging operator. Define a new operation
$\{-, -\} : A\times A\rightarrow A$ by
$$\{x, y\}:=[\alpha(x), y].$$
Then $(A, \{-, -\}, \varepsilon)$ is a Lie color algebra. 
\end{example}
 We need the following definition for producing other examples of Leibniz color algebras.
\begin{definition}
A  Rota-Baxter Leibniz color algebra of weight $\lambda\in\mathbb{K}$ is a Leibniz color 
algebra $(L, [-, -], \varepsilon)$ together with an even linear map $R : L\rightarrow L$ that satisfies the {\it Rota-Baxter identity}
\begin{eqnarray}
[R(x), R(y)] = R\Big([R(x), y] + [x, R(y)] +\lambda [x, y]\Big),\label{rbhl}
\end{eqnarray}
for all $x, y\in\mathcal{H}(L)$.
\end{definition}
We have a similar definition for Rota-Baxter Lie color algebras.

\begin{example}\label{ed}
 Let $(L, [-, -], \varepsilon, R)$ be a Rota-Baxter Leibniz color algebra of weight $\lambda$. Then
$L(\lambda)=(L, [-, -]_\lambda, \varepsilon, R)$
is also a Rota-Baxter  Leibniz color algebra of weight $\lambda$ with respect to the bracket
$$[x, y]_\lambda=[R(x), y]+[x, R(y)]+\lambda[x, y]$$
for all $x, y\in\mathcal{H}(L)$.
\end{example}

 Now we give the below definition.
\begin{definition}\cite{BD}
 A $G$-graded non-associative algebra $(A, \cdot)$ with a bicharacter $\varepsilon : G\otimes G\rightarrow \mathbb{K}^*$
 satisfying, for all $x, y, z\in \mathcal{H}(A)$,
\begin{eqnarray}
(x\cdot y)\cdot z-x\cdot(y\cdot z)=\varepsilon(x, y)\Big((y\cdot x)\cdot z-y\cdot(x\cdot z)\Big),
\end{eqnarray}
is called a left-symmetric color algebra.
\end{definition}

The next result asserts that the $\varepsilon$-commutator of any left-symmetric color algebra leads to a Lie color algebra. 
\begin{theorem}\label{ghls-ghl}\cite{BD}
Let $(A, \cdot, \varepsilon)$ be a left-symmetric color algebra. Then, for any $x, y\in\mathcal{H}(A)$, the commutator 
\begin{eqnarray}
 [x, y]=x\cdot y-\varepsilon(x, y)y\cdot x,
\end{eqnarray}
makes $A$ to become a Lie color algebra.
\end{theorem}
\begin{lemma}\label{hde}\cite{B2}
 Let $(L, [-, -], \varepsilon, R)$ be a Rota-Baxter Lie color algebra of weight $0$. Then $L$ is a left-symmetric color algebra
 with $$x\cdot y=[R(x), y],$$
for $x, y\in\mathcal{H}(L)$.
\end{lemma}
The below proposition comes from Theorem \ref{ghls-ghl}  and  Lemma  \ref{hde}.
\begin{proposition}
 Let $(L, [-, -], \varepsilon, R)$ be a Rota-Baxter Lie color algebra of weight $0$. Then $(L, \{-, -\}, \varepsilon, R)$ is also a Rota-Baxter 
 Lie color algebra of weight $0$ with respect to the bracket
$$\{x, y\}=[R(x), y]-\varepsilon(x, y)[R(y), x]$$
for each $x, y\in\mathcal{H}(L)$.
\end{proposition}
In what follows, we extend to color case some results on superalgebras \cite{A}. The different results are also proved similarly.

We call Rota-Baxter left-symmetric color algebra of weight $\lambda$, a left-symmetric color algebra endowed with an even linear map 
$R$ and satisfying the Rota-Baxter identity (\ref{rbi}).
\begin{lemma}\label{rbs}
 Let $(A, \cdot, \varepsilon, R)$ be a Rota-Baxter left-symmetric color algebra of weight $0$. Then $(A, \ast, \varepsilon, R)$ is also a
Rota-Baxter left-symmetric color algebra with respect to 
$$x\ast y:=R(x)\cdot y-\varepsilon(x, y)y\cdot R(x).$$
\end{lemma}
\begin{proposition}
 Let $(A, \cdot, \varepsilon, R)$ be a Rota-Baxter left-symmetric color algebra of weight $0$. Then $(A, [-, -], \varepsilon)$ is Lie color
algebra, where
$$[x, y]:=R(x)\cdot y+x\cdot R(y)-\varepsilon(x, y)\Big(R(y)\cdot x+y\cdot R(x)\Big).$$
\end{proposition}
\begin{proof}
 It follows from Theorem \ref{ghls-ghl} and Lemma \ref{rbs}.
\end{proof}
\begin{lemma}\label{rba}
  Let $(A, \cdot, \varepsilon, R)$ be a Rota-Baxter associative color algebra of weight $\lambda=-1$. Let us define
the even bilinear operation $\ast : A\otimes A\rightarrow A$ on homogeneous elements $x, y\in A$ by 
$$x\ast y:=R(x)\cdot y-\varepsilon(x, y)y\cdot R(x)-x\cdot y.$$
Then $(A, \ast, \varepsilon)$ is a left-symmetric color algebra.
\end{lemma}
\begin{proposition}
 Let $(A, \cdot, \varepsilon, R)$ be a Rota-Baxter associative color algebra of weight $-1$. Then $(A, [-, -], \varepsilon)$ is Lie color
algebra, with respect to
$$[x, y]:=R(x)\cdot y+x\cdot R(y)-x\cdot y-\varepsilon(x, y)\Big(R(y)\cdot x+y\cdot R(x)-y\cdot x\Big).$$
\end{proposition}
\begin{proof}
 The proof follows from Theorem \ref{ghls-ghl} and Lemma \ref{rba}.
\end{proof}
\begin{definition} \cite{B2}
 A color algebra  $(A, \cdot, \varepsilon)$ is said to be a Lie  admissible color algebra if, for any hogeneous elements
$x, y\in A$, the bracket $[-, -]$ defined
by $$[x, y]=x\cdot y-\varepsilon(x, y)y\cdot x$$
satisfies the $\varepsilon$-Jacobi identity.
\end{definition}
\begin{example}
 Any associative color algebra and any Lie color  algebra are Lie admissible color algebras \cite{B2}.
\end{example}
\begin{lemma}\label{rbad}
 Let $(A, \cdot, \varepsilon, R)$ be a Rota-Baxter Lie admissible color algebra of weight $0$. Then 
 $(A, \ast, \varepsilon)$ is a Rota-Baxter left-symmetric color algebra with respect to
$$x\ast y:=R(x)\cdot y-\varepsilon(x, y)y\cdot R(x),$$
for any $x, y\in \mathcal{H}(A)$.
\end{lemma}
It is known that every Rota-Baxter associative color algebra is color Lie admissible. Then we have the following proposition.
\begin{proposition}
 Let $(A, \cdot, \varepsilon, R)$ be a Rota-Baxter associative color algebra of weight $0$. Equipped with the bracket
$$[x, y]:=R(x)\cdot y+x\cdot R(y)-\varepsilon(x, y)\Big(R(y)\cdot x+y\cdot R(x)+y\cdot x\Big),$$
 $A$ becomes a Lie color algebra.
\end{proposition}
\begin{proof}
 The proof follows from Theorem \ref{ghls-ghl} and Lemma \ref{rbad}.
\end{proof}
\begin{definition}\label{hplad}\cite{B2}
 A post-Lie color algebra $(L, [-, -], \cdot, \varepsilon)$ is a Lie color  algebra  $(L, [-, -], \varepsilon)$ 
together with an even bilinear map $\cdot : L\otimes L\rightarrow L$ such that
\begin{eqnarray}
 z\cdot [x, y]-[z\cdot x, y]-\varepsilon(z, x)[x, z\cdot y]=0,\qquad\qquad\qquad\qquad \label{pl4}\\
z\cdot(y\cdot x)-\varepsilon(z, y)y\cdot(z\cdot x)+\varepsilon(z, y)(y\cdot z)\cdot x
-(z\cdot y)\cdot x+\varepsilon(z, y)[y, z]\cdot x=0,\label{pl3}
\end{eqnarray}
for any $x, y, z\in\mathcal{H}(L)$.
\end{definition}
 \begin{lemma}\label{chla}\cite{B2}
 Let $(L, [-, -], \cdot, \varepsilon)$ be a post-Lie color algebra. Define the even bilinear product $\ast : L\times L\rightarrow L$ as 
\begin{eqnarray}
 x\ast y:=x\cdot y+\frac{1}{2}[x, y].\nonumber
\end{eqnarray}
Then $(L, \ast, \varepsilon)$ is a Lie  admissible color algebra.
\end{lemma}
According to Lemma \ref{chla} and Theorem \ref{ghls-ghl}, we have :
\begin{proposition} \label{rbads}
 Let $(L, [-, -], \cdot, \varepsilon, R)$ be a Rota-Baxter post-Lie color algebra. Then $A$ is a left-symmetric color algebra with respect to
\begin{eqnarray}
 x\circ y:= R(x)\ast y-\varepsilon(x, y)y\ast R(x)=[R(x), y]+\{R(x), y\},\nonumber
\end{eqnarray}
where $\{-, -\}$ is the $\varepsilon$-commutator bracket.
\end{proposition}
\begin{proposition}
 Let $(L, [-, -], \cdot, \varepsilon)$ be a Rota-Baxter post-Lie color algebra. Then $(A, \{-, -\}, \varepsilon)$ is a Lie color algebra, where
$$\{x, y\}:=[R(x), y]+R(x)\cdot y+ x\cdot R(y)-\varepsilon(x, y)\Big([R(y), x]+R(y)\cdot x+ y\cdot R(x)\Big).$$
\end{proposition}
\begin{proof}
 The proof follows from Theorem \ref{ghls-ghl} and Proposition \ref{rbads}.
\end{proof}

Let us continue to give other examples of Leibniz color algebras.
\begin{example}\cite{BD}
 A left-symmetric color dialgebra is a $G$-graded vector space $S$ equipped with a bicharacter $\varepsilon : G\otimes G\rightarrow \mathbb{K}^*$
 on $G$ and two even bilinear products $\dashv : S\times S\rightarrow S$ and $\vdash : S\times S\rightarrow S$ satisfying the identities
\begin{eqnarray}
 x\dashv(y\dashv z)&=&x\dashv(y\vdash z),\label{als1}\nonumber\\
(x\vdash y)\vdash z&=&(x\dashv y)\vdash z,\label{als2}\nonumber\\
x\dashv(y\dashv z)-(x\dashv y)\dashv z&=&\varepsilon(x, y)\Big(
y\vdash(x\dashv z)-(y\vdash x)\dashv z\Big),\nonumber\label{als3}\\
x\vdash(y\vdash z)-(x\vdash y)\vdash z&=&\varepsilon(x, y)\Big(
y\vdash(x\vdash z)-(y\vdash x)\vdash z\Big),\nonumber\label{als4}
\end{eqnarray}
for all $x, y, z\in \mathcal{H}(S)$.
 Then the commutator bracket defined by
\begin{eqnarray}
 [x, y]:=x\dashv y-\varepsilon(x, y)y\vdash x,\nonumber
\end{eqnarray}
defines a structure of Leibniz color algebra on $S$.
\end{example}
\begin{example}\label{H}\cite{BD}
 An associative color dialgebra is a quadruple $(D, \dashv, \vdash, \varepsilon)$, where $D$ is a $G$-graded vector space, 
$\dashv, \vdash : D\otimes D\rightarrow D$ are even bilinear maps and $\varepsilon : G\otimes G\rightarrow \mathbb{K}^*$ is a bicharacter
 such that the following  five axioms 
 \begin{eqnarray}
(x\vdash y)\dashv z&\stackrel{}{=}&x\vdash(y\dashv z), \nonumber\\
x\dashv (y\dashv z)&\stackrel{}{=}&(x\dashv y)\dashv z\stackrel{}{=}x \dashv(y\vdash z),\nonumber\\
(x\vdash y)\vdash z&\stackrel{}{=}&x\vdash(y\vdash z)\stackrel{}{=}(x\dashv y)\vdash z,\nonumber
 \end{eqnarray}
are satisfied for $x, y, z\in \mathcal{H}(D)$.
Define the even bilinear map $[-, -] : D\otimes D\rightarrow D$
 by setting
\begin{eqnarray}
 [x, y]:=x\dashv y-\varepsilon(x, y) y\vdash x.\nonumber
\end{eqnarray}
Then $(D, [-, -], \varepsilon)$ is a Leibniz color algebra.
\end{example}
\begin{example}
 Let $(A, \cdot, \varepsilon)$ be an associative color algebra and $\alpha : A\rightarrow A$ be an averaging operator. Define new operations on 
$A$ by 
$$x\dashv y:=x\cdot\alpha(y)\quad\mbox{and}\quad x\vdash y:=\alpha(x)\cdot y.$$
Then $(A, \dashv, \vdash, \varepsilon)$ is an associative color dialgebra. And from Example \ref{H}, $A$ is a Leibniz color algebra with respect
to the bracket $[x, y]:=x\cdot\alpha(y)-\varepsilon(x, y)\alpha(y)\cdot x$.
\end{example}

\section{Ternary Leibniz color algebras and  associative color trialgebras}
This section is devoted to the construction of some ternary Leibniz color algebras.
\subsection{Ternary Leibniz color algebras}
\begin{definition}
 A ternary Leibniz color algebra is a $G$-graded vector space $A$ over a field $\mathbb{K}$ equipped with a bicharacter 
$\varepsilon : G\otimes G\rightarrow \mathbb{K}^*$ and an even trilinear operation $[-, -, -] : A\otimes A\otimes A\rightarrow A$ (i.e.
$[x, y, z]\subseteq  A_{x+y+z}$ whenever  $x, y, z\in \mathcal{H}(A)$) satisfying the following  {\it ternary $\varepsilon$-Nambu identity} :
\begin{eqnarray}
 [[x, y, z], t, u]=[x, y, [z, t, u]]+\varepsilon(z, t+u)[x, [y, t, u], z]+\varepsilon(y+z, t+u)[[x, t, u], y, z]\label{lci}
\end{eqnarray}
for any $x, y, z, t, u\in \mathcal{H}(A)$.\\
If the trilinear map $[-, -, -]$ is $\varepsilon$-skew-symmetric for any pair of variables, then $(A, [-, -, -], \varepsilon)$ is 
said to be a ternary Lie color algebra \cite{TZ}.
\end{definition}
It is proved in (\cite{JM}, Proposition 3.2) that any Leibniz algebra is also a Leibniz $n$-algebra. We prove the analog for graded case and $n=3$.
That is one can get ternary Leibniz color algebras from Leibniz color algebras.
\begin{theorem}\label{ll3}
Let $(L, [-, -], \varepsilon)$ be a Leibniz color algebra. Then $L$ is a ternary Leibniz color algebra with respect to the bracket
$$[x, y, z]:=[x, [y, z]],$$
for  any $x, y, z\in \mathcal{H}(L)$.
\end{theorem}
\begin{proof}
 Applying twice relation (\ref{cpa}), for any $x, y, z\in \mathcal{H}(L)$, we have
\begin{eqnarray}
 [[x, y, z], t, u]&=&[[x, [y, z]], [t, u]]\nonumber\\
&=&[x,  [[y, z], [t, u]]]+\varepsilon(y+z, t+u)[[x, [t, u]]], [y, z]]\nonumber\\
&=&[x, [y, [z, [t, u]]]]+\varepsilon(z, t+u)[x, [[y, [t, u]], z]]+\varepsilon(y+z, t+u)[[x, [t, u]]], [y, z]]\nonumber\\
&=&[x, y, [z, t, u]]+\varepsilon(z, t+u)[x, [y, t, u], z]+\varepsilon(y+z, t+u)[[x, t, u], y, z].\nonumber
\end{eqnarray}
This completes the proof.
\end{proof}
\begin{example}
 Examples (\ref{a3c})-(\ref{ed}) provide various examples of ternary Leibniz color algebras.
\end{example}

\begin{remark}
 A derivation of $(A, [-, -], \varepsilon)$ is also a derivation of $(A, [-, [-, -]], \varepsilon)$.
\end{remark}

Here, we recall the definition of Leibniz-Poisson color algebras \cite{BD} and give some examples,
 which will also give examples for Theorem \ref{ll}.

\begin{definition} 
 A non-commutative Leibniz-Poisson color algebra is a $G$-graded vector space $P$ together with two even bilinear maps 
$[-, -] : P\otimes P\rightarrow P$ and $\cdot : P\otimes P\rightarrow P$ and a bicharacter $\varepsilon : G\otimes G\rightarrow \mathbb{K}^*$
  such that
\begin{enumerate}
\item [1)] $(P, \cdot, \varepsilon)$ is an associative color algebra,
 \item [2)]  $(P, [-, -], \varepsilon)$ is a Leibniz color algebra,
\item [3)] and the following {\it right Leibniz} identity :
 \begin{eqnarray}
  [x\cdot y, z]=x\cdot [y, z]+\varepsilon(y, z) [x, z]\cdot y \label{comp}
 \end{eqnarray}
holds, for all $x, y, z\in \mathcal{H}(P)$.
\end{enumerate}
\end{definition}
\begin{remark}
1) When the color associative product $\cdot$ is $\varepsilon$-commutative i.e. $x\cdot y=\varepsilon(x, y)y\cdot x$, 
then $(P, \cdot, [-, -], \varepsilon)$ is said to be
a commutative Leibniz-Poisson color algebra.\\
2) A non-commutative Leibniz-Poisson algebra $(P, \cdot, [-, -], \varepsilon)$ in which the bracket $[-, -]$ is 
$\varepsilon$-skew symmetric, is called a non-commutative Poisson color algebra.\\
3) Whenever the color associative product $\cdot$ is $\varepsilon$-commutative and the bracket $[-, -]$ is $\varepsilon$-skew-symmetric,
then $(P, \cdot, [-, -], \varepsilon)$ is named a commutative Poisson color algebra.
\end{remark}
\begin{example}\label{laa1}
Let $(D, \dashv, \vdash, \varepsilon)$ be an associative color dialgebra. Then $(D, \dashv, [-, -], \varepsilon)$ is a non-commutative
 Leibniz-Poisson color algebra with the bracket
\begin{eqnarray}
 [x, y]:=x\dashv y-\varepsilon(x, y) y\vdash x,\nonumber
\end{eqnarray}
for all $x, y\in \mathcal{H}(D)$.
\end{example}

\begin{proposition}\label{laa2}\cite{B}
 Let $(A, \cdot, \varepsilon)$ be an associative color algebra.
 Then $(A, \cdot, [-, -]=\cdot-\varepsilon\;\cdot^{op}, \varepsilon)$ is a non-commutative Leibniz-Poisson  color algebra , where
$x\cdot^{op} y=y\cdot x$, for any $x, y\in \mathcal{H}(A)$.
\end{proposition}

\subsection{Associative color trialgebras}
We now introduce  color trialgebras and establish their relationship with ternary Leibniz color algebras.
\begin{definition}
 An associative color trialgebra is a $G$-graded vector space $A$ equipped with a bicharacter $\varepsilon : G\otimes G\rightarrow \mathbb{K}^*$
and three even binary associative operations $\dashv, \perp, \vdash :
A\otimes A\rightarrow A$ (called left, middle and right respectively), satisfying the following relations :
\begin{eqnarray}
 (x\dashv y)\dashv z&=&x\dashv(y\vdash z)=x\dashv(y\perp z)\\
(x\vdash y)\dashv z&=&x\vdash(y\dashv z)\\
(x\dashv y)\vdash z&=&x\vdash(y\vdash z)=(x\perp y)\vdash z\\
(x\perp y)\dashv z&=&x\perp(y\dashv z)\\
(x\dashv y)\perp z&=&x\perp(y\vdash z)\\
(x\vdash y)\perp z&=&x\vdash(y\perp z)
\end{eqnarray}
for  all $x, y, z\in \mathcal{H}(A)$.
\end{definition}
\begin{example}
 Any associative color algebra $(A, \cdot, \varepsilon)$ is a color trialgebra with $\cdot=\dashv=\perp=\vdash$.
\end{example}
\begin{example}
 Any associative color dialgebra is a color trialgebra with trivial middle product.
\end{example}

\begin{example}
 If $(A, \dashv, \perp, \vdash, \varepsilon)$ is a color trialgebra, then so is $(A, \dashv', \perp', \vdash', \varepsilon)$, where
$$x\dashv' y:= y\vdash x, \quad x\perp' y:= y\perp x, \quad x\vdash' y:= y\dashv x.$$
\end{example}
The following proposition connects color trialgebras to non-commutative Leibniz-Poisson color algebras.
 As examples \ref{laa1} and \ref{laa2}, it will also gives a construction of non-commutative ternary Leibniz-Nambu-Poisson color
algebras.
\begin{proposition}
 Let $(A, \dashv, \perp, \vdash, \varepsilon)$ be a color trialgebra. Then $(A, \cdot, [-, -], \varepsilon)$ is a non-commutative
 Leibniz-Poisson color algebra with respect to the operations
$$ x\cdot y:= x\perp y \quad [x, y]=x\dashv y-\varepsilon(x, y)x\vdash y, $$
for  all $x, y, z\in \mathcal{H}(A)$.
\end{proposition}

\begin{theorem}
 Let $(A, \dashv, \perp, \vdash, \varepsilon)$ be a color trialgebra. Then $A$ is a ternary Leibniz color algebra with respect to the bracket
\begin{eqnarray}
 [x, y, z]:=x\dashv(y\perp z-\varepsilon(y, z)z\perp y)-\varepsilon(x, y+z)(y\perp z-\varepsilon(y, z)z\perp y)\vdash x,\nonumber
\end{eqnarray}
for  all $x, y, z\in \mathcal{H}(A)$.
\end{theorem}
\begin{proof}
 It follows from a straightforward computation.
\end{proof}
\begin{corollary}
 Let $(A, \cdot, \varepsilon)$ be an associative color algebra. Then $(A, [-, -, -], \varepsilon)$ is a ternary Leibniz color algebra, where
$$[x, y, z]:=x\cdot(y\cdot z-\varepsilon(y, z)z\cdot y)-\varepsilon(x, y+z)(y\cdot z-\varepsilon(y, z)z\cdot y)\cdot x.$$
\end{corollary}
Observe that this ternary Leibniz color algebra coincides with the one constructed from Example \ref{a3c} and Theorem \ref{ll3}.

Now we give some results on tridendriform color algebras and show that tridendriform color algebra leads to non-commutative Leibniz-Poisson
color algebras.
\begin{definition}
Let $q\in\bf \mathbb{K}$. A $q$-tridendriform color algebra is a quintuple $(A, \dashv, \vdash, \cdot, \varepsilon)$ consisting of
a bicharacter $\varepsilon : G\times G\rightarrow\mathbb{K}^*$ and a vector space $A$ on which the operations
$\dashv, \vdash, \cdot : A\otimes A\rightarrow A$ are three bilinear maps satisfying :
\begin{eqnarray}
(x\dashv y)\dashv z&=&x\dashv(y\dashv z+\varepsilon(z, y) y\vdash z+q\varepsilon(z, y) y\cdot z), \label{initf}\\
(x\vdash y)\dashv z&=&\varepsilon(z,x) x\vdash(y\dashv z),\label{initfh}\\
x\vdash(y\vdash z)&=&(\varepsilon(x, y) x\dashv y+x\vdash y+qx\cdot y)\vdash z,\\
(x\dashv y)\cdot z&=&\varepsilon(y, x) x\cdot(y\vdash z),\\
(x\vdash y)\cdot z&=&x\vdash(y\cdot z),\\
(x\cdot y)\dashv z&=&\varepsilon(z, x) x\cdot(y\dashv z),\\
(x\cdot y)\cdot z&=&x\cdot(y\cdot z),\label{initl1}
\end{eqnarray}
for $x, y, z\in \mathcal{H}(A)$.
\end{definition}
\begin{remark}
 \begin{enumerate}
  \item [1)] Whenever $q=1$, we recover the usual tridendriform color algebra \cite{B2}.
\item [2)] For any $q\neq0$, $(A, \dashv, \vdash, q^{-1}\cdot, \varepsilon)$ turn to tridendriform color algebra whenever
 $(A, \dashv, \vdash, \cdot, \varepsilon)$ is a $q$-tridendriform color algebra. Conversely, if 
$(A, \dashv, \vdash, \cdot, \varepsilon)$ is a tridendriform color algebra, then $(A, \dashv, \vdash, q\cdot, \varepsilon)$ 
is a $q$-tridendriform color algebra.
\item [3)] We recover dendriform color algebras \cite{BD} when the associative product is identically null.
 \end{enumerate}
\end{remark}
\begin{definition}

An even linear map $R : A\rightarrow A$ on a $q$-tridendriform algebra color $(A, \dashv, \vdash, \cdot, \varepsilon)$ is said to be a 
Rota-Baxter operator of weight $\lambda\in\mathbb{K}$ if each operation satisfies
{\it the Rota-Baxter identity} (\ref{rbi}).
In this case, the quintuple $(A, \dashv, \vdash, \cdot, R)$ is called {\it Rota-Baxter $q$-tridendriform color  algebra of weight $\lambda$}.
\end{definition}
\begin{remark}
 We have a similar definition for {\it Rota-Baxter associative color trialgebras},  which are associative color trialgebras endowed with
 an even linear operator that satisfies the {\it Rota-Baxter identity} for each product. 
\end{remark}
\begin{proposition}\label{car1}
 Let $(A, \cdot, \varepsilon, R)$ be a Rota-Baxter associative color algebra of weight $\lambda$. Let us define the operations $\dashv, \vdash$ and 
$\bullet$ on $A$ by 
$$x\dashv y:=x\cdot R(y), \quad x\vdash y:=\varepsilon(x, y)R(x)\cdot y\quad\mbox{and}\quad x\bullet y:=\lambda\varepsilon(x, y)x\cdot y.$$
Then $(A, \dashv, \vdash, \bullet, \varepsilon)$ is a tridendriform color algebra.
\end{proposition}
\begin{proof}
 The proof follows from a simple calculation. For instance, for axiom (\ref{initfh}), one has :
\begin{eqnarray}
 (x\dashv y)\vdash-\varepsilon(z, x)x\vdash (y\dashv z)
&=&\varepsilon(x, y)(R(x)\cdot y)\cdot R(z)-\varepsilon(z, x)\varepsilon(x, y+z) R(x)\cdot(y\cdot R(z))\nonumber\\
&=&\varepsilon(x, y)\Big((R(x)\cdot y)\cdot R(z)- R(x)\cdot(y\cdot R(z)) \Big).\nonumber
\end{eqnarray}
The last line vanishes by associativity.
\end{proof}

\begin{proposition}
 Let $(A, \cdot, \varepsilon, R)$ be a Rota-Baxter associative color algebra of weight $\lambda$.
Define the operations $\dashv$ and $\vdash$ by 
$$x\dashv y:=x\cdot R(y)+\lambda x\cdot y \quad\mbox{and}\quad x\vdash y:=\varepsilon(x, y)R(x)\cdot y.$$
Then $(A, \dashv, \vdash, \varepsilon)$ is a dendriform color algebra.
\end{proposition}
\begin{proof}
The axioms are checked as in the proof of Proposition \ref{car1}.
\end{proof}
\begin{corollary}
 Let $(A, \cdot, \varepsilon, R)$ be a Rota-Baxter associative color algebra of weight $0$. We define the even bilinear operations 
$\dashv : A\times A\rightarrow A$ and $\vdash : A\times A\rightarrow A$ on $A$ by
$$x\dashv y:=x\cdot R(y)\quad\mbox{and}\quad x\vdash y:=\varepsilon(x, y)R(x)\cdot y.$$
Then $(A, \dashv, \vdash, \varepsilon)$ is a dendriform color algebra.
\end{corollary}
The below lemma associates associative color algebras to tridendriform color algebras.
\begin{lemma}\cite{B2}\label{car22}
Let $(A, \dashv, \vdash, \cdot, \varepsilon)$ be a  tridendriform color algebra. Then $(A, \ast, \varepsilon)$ is
an associative color algebra, where
 $x\ast y=x\vdash y+\varepsilon(x, y)x\dashv y+ x\cdot y$.
\end{lemma}
\begin{theorem}\label{car2}
 Let $(A, \dashv, \vdash, \bullet, \varepsilon)$ be a tridendriform color algebra. Then $(A, \ast, [-, -], \varepsilon)$ is a non-commutative
 Leibniz-Poisson color algebra, where
$$x\ast y:=x\vdash y+\varepsilon(x, y)x\dashv y+x\cdot y\quad\mbox{and}\quad [x, y]:=x\ast y-\varepsilon(x, y)y\ast x.$$
\end{theorem}
\begin{proof}
 The proof follows from Lemma \ref{car22} and Proposition \ref{laa2}.
\end{proof}

\begin{corollary}
  Let $(A, \cdot, \varepsilon, R)$ be a Rota-Baxter associative color algebra of weight $\lambda$. Then $A$ is a non-commutative Leibniz-Poisson color
algebra with respect to the products :
$$x\ast y:=\varepsilon(x, y)\Big( R(x)\cdot y+x\cdot R(y)+\lambda x\cdot y\Big)\quad\mbox{and}\quad [x, y]:=x\ast y-\varepsilon(x, y)y\ast x.$$
\end{corollary}
\begin{proof}
 It comes from propositions \ref{car1} and \ref{car2}.
\end{proof}

 We end this subsection by establishing a connection between trialgebras and $-1$-tridendriform algebras (with trivial grading).
\begin{theorem}
 Let $(A, \dashv, \perp, \vdash)$ be an associative trialgebra. Then $(A, \dashv, \vdash, \perp)$ is a $-1$-tridendriform algebra.
\end{theorem}
\begin{proof}
 The proof comes from both definitions.
\end{proof}

\begin{corollary}\label{ta}
Let $(A, \dashv, \perp, \vdash)$ be an associative trialgebra. Then $A$ is an associative algebra with respect to the multiplication
$\ast : A\otimes A\rightarrow A$ :
$$x\ast y=x\dashv y+x\vdash y-x\perp y$$
for any $x, y\in A$.
\end{corollary}
\begin{corollary}
 Let $(A, \dashv, \perp, \vdash)$ be an associative trialgebra. Then $(A, \ast, [-, -])$ is a non-commutative Leibniz-Poisson algebra with 
$$x\ast y=x\dashv y+x\vdash y-x\perp y\quad{and}\quad [x, y]=x\ast y-y\ast x$$
for any $x, y\in A$.
\end{corollary}

The below corollaries are based on some results of \cite{MD} and the following remark.
\begin{remark}
 If $R$ is a Rota-Baxter operator of weight $\lambda\in\mathbb{K}$ on the trialgebra $(A, \dashv, \perp, \vdash)$, it is also a Rota-Baxter 
operator of weight $\lambda\in\mathbb{K}$ on the associative algebra $(A, \ast)$ of Corollary \ref{ta}.
\end{remark}

\begin{corollary}
 Let $(A, \dashv, \perp, \vdash, R)$ be a Rota-Baxter trialgebra of weight $0$. Then $(A, \star)$ is a left-symmetric algebra with 
$$x\star y=R(x)\ast y-y\ast R(x)\quad{and}\quad  x\ast y=x\dashv y+x\vdash y-x\perp y$$
for all $x, y\in A$.
\end{corollary}
\begin{corollary}
Let $(A, \dashv, \perp, \vdash, R)$ be a Rota-Baxter trialgebra of weight $-1$. Then 
$(A, \star)$ is an associative algebra with 
$$x\star y=R(x)\ast y-y\ast R(x)-x\ast y\quad{and}\quad  x\ast y=x\dashv y+x\vdash y-x\perp y.$$ 
\end{corollary}
\subsection{Color Lie triple systems}
\begin{definition}
 A color Lie triple system is a $G$-graded vector space $A$ over a field $\mathbb{K}$ equipped with a bicharacter 
$\varepsilon : G\times G\rightarrow\mathbb{K}^*$ and an even trilinear bracket which satisfies the identity (\ref{lci}), instead of skew-symmetry,
satisfies the conditions
\begin{eqnarray}
 &&\qquad\qquad\qquad[x, y, z]=-\varepsilon(y, z)[x, z, y], \quad(\mbox{\it right}\; \varepsilon\mbox{\it -skew-symmetry})\label{rss}\\
&&\varepsilon(z, x)[x, y, z]+\varepsilon(x, y)[y, z, x]+\varepsilon(y, z)[z, x, y]=0,
\;\;(\mbox{\it ternary}\;\;\varepsilon\mbox{\it -Jacobi identity} )\label{ltsji}
\end{eqnarray}
for each $x, y, z\in\mathcal{H}(A)$.
\end{definition}
\begin{theorem}
 Let $(A, [-, -], \varepsilon)$ be a Lie color algebra. Define the even trilinear map $[-, -, -] : A\otimes A\otimes A\rightarrow A$ by
$$[x, y, z]:=[x, [y, z]]$$
for any $x, y, z\in\mathcal{H}(A)$.
Then $(A, [-, -, -], \varepsilon)$ is a color Lie triple system.
\end{theorem}
\begin{proof}
The right skew-symmetry (\ref{rss}) is immediate. The ternary $\varepsilon$-Jacobi identity (\ref{ltsji}) follows from $\varepsilon$-Jacobi
 identity. And Theorem \ref{ll3} completes the proof.
\end{proof}
\begin{corollary}
 Let $(A, \cdot, \varepsilon)$ be an associative color algebra. Then $A$ is a color Lie triple system with respect to the bracket
$$[x, y, z]:=x\cdot (y\cdot z)-\varepsilon(y, z)x\cdot (z\cdot y)
-\varepsilon(x, y+z)(y\cdot z)\cdot x+ \varepsilon (x, y+z)\varepsilon(y, z)(z\cdot y)\cdot x.$$
\end{corollary}
Next we introduce color Jordan triple systems and study their connection with color Lie triple systems.
\begin{definition}
 A color Jordan triple system is a triple $(J, [-, -, -], \varepsilon)$ in which
 $J$ is a $G$-graded vector space over a field $\mathbb{K}$, $\varepsilon : G\times G\rightarrow\mathbb{K}^*$ is a bicharacter and 
$[-, -, -] : A\otimes A\otimes A\rightarrow A$ is an even trilinear map satisfying
\begin{eqnarray}
 [x, y, z]&=&\varepsilon(x, y)\varepsilon(x, z)\varepsilon(y, z)[z, y, x]\quad\mbox{\it (outer-}\varepsilon\mbox{\it -symmetry}) \label{oss}
\end{eqnarray}
and {\it color Jordan triple identity}
\begin{eqnarray}
{[[x, y, z], t, u]}&=&[x, y, [z, t, u]]-\varepsilon(z, t+u)\varepsilon(t, u)[x, [y, u, t], z]+\varepsilon(y+z, t+u)[[x, t, u], y, z],
\label{jts}
\end{eqnarray}
for any $x, y, z\in\mathcal{H}(J)$.
\end{definition}
\begin{example}
 Let $(A, \cdot, \varepsilon)$ be an associative color algebra. Then $(A, [-, -, -], \varepsilon)$ is a color Jordan triple system with respect
to the triple product
$$[x, y, z]:=x\cdot y\cdot z+\varepsilon(x, y)\varepsilon(x, z)\varepsilon(y, z)z\cdot y\cdot x.$$
\end{example}
\begin{example}
Let $(A, \cdot, \varepsilon)$ be an associative color algebra and $\theta : A\rightarrow A$ be an even linear map on $A$ satisfying $\theta^2=Id_A$ and 
$\theta(x\cdot y)=\varepsilon(x, y)\theta(y)\cdot\theta(x)$ for any $x, y\in \mathcal{H}(A)$. Then $(A, [-, -, -], \varepsilon)$ is a color Jordan
triple system with the triple product
$$[x, y, z]:=x\cdot\theta(y)\cdot z+\varepsilon(x, y)\varepsilon(x, z)\varepsilon(y, z)z\cdot\theta(y)\cdot x.$$
\end{example}
We have the following result.
\begin{theorem}
 Let $(J, [-, -, -], \varepsilon)$ be a color Jordan triple system. Define the triple product 
$$\{x, y, z\}:=[x, y, z]-\varepsilon(y, z)[x, z, y]$$
for any $x, y, z\in\mathcal{H}(J)$. Then $L(J)=(J, \{-, -, -\}, \varepsilon)$ is a color Lie triple system.
\end{theorem}
\begin{proof}
The {\it right $\varepsilon$-skew-symmetry} is immediate. The {\it ternary $\varepsilon$-Jacobi identity} follows from (\ref{oss}). 
It remains to check the {\it ternary $\varepsilon$-Nambu identity} (\ref{lci}) for $\{-, -, -\}$. For any $x, y, z\in\mathcal{H}(J)$, we have
\begin{eqnarray}
 \{\{x, y, z\}, t, u\}
&=&[\{x, y, z\}, t, u]-\varepsilon(t, u)[\{x, y, z\}, u, t]\nonumber\\
&=&[([x, y, z]-\varepsilon(y, z)[x, z, y]), t, u]-\varepsilon(t, u)[([x, y, z]-\varepsilon(y, z)[x, z, y]), u, t]\nonumber\\
&=&[[x, y, z], t, u]-\varepsilon(y, z)[[x, z, y], t, u]-\varepsilon(t, u)[[x, y, z], u, t]
+\varepsilon(t, u)\varepsilon(y, z)[[x, z, y], u, t].\nonumber
\end{eqnarray}
Similarly,
\begin{eqnarray}
 \{x, y, \{z, t, u\}\}&=& [x, y, [z, t, u]]-\varepsilon(t, u)[x, y, [z, u, t]]-\varepsilon(y, z+t+u)[x, [z, t, u], y]
+\varepsilon(y, z+t+u)\varepsilon(t, u),\nonumber\\
\{x, \{y, t, u\}, z\}&=&[x, [y, t, u], z]-\varepsilon(t, u)[x, [y, u, t], z]-\varepsilon(y+t+u, z)[x, z, [y, t, u]]\nonumber\\
&&+\varepsilon(y+t+u, z)\varepsilon(t, u)[x, z, [y, u, t]],\nonumber\\
\{\{x, t, u\}, t, u\}&=&[[x, t, u], y, z]-\varepsilon(t, u)[[x, u, t], y, z]-\varepsilon(y, z)[[x, t, u], z, y]
+\varepsilon(y, z)\varepsilon(t, u)[[x, u, t], z, y]\nonumber.
\end{eqnarray}
Then, using axiom (\ref{jts}), {\it ternary $\varepsilon$-Nambu identity} (\ref{lci}) holds for the bracket $\{-, -, -\}$.
\end{proof}

\section{Ternary Leibniz-Nambu-Poisson color algebras}
We introduce non-commutative ternary Leibniz-Nambu-Poisson color algebras and give some procedures of construction.

\begin{definition}
 A non-commutative ternary Leibniz-Nambu-Poisson color algebra is a quadruple $(A, \cdot, [-, -, -], \varepsilon)$ in which
\begin{enumerate}
 \item [1)] $(A, \cdot, \varepsilon)$ is an associative color algebra,
\item [2)] $(A, [-, -, -], \varepsilon)$ is a ternary Leibniz color algebra,
\item [3)] and the following {\it right ternary Leibniz rule }
\begin{eqnarray}
 [x\cdot y, z, t]=x\cdot [y, z, t]+\varepsilon(y, z+t)[x, z, t]\cdot y \label{cal3} 
\end{eqnarray}
holds for any $x, y, z, t\in \mathcal{H}(A)$.
\end{enumerate}
If in addition, the product $\cdot$ is $\varepsilon$-commutative, then the non-commutative ternary Leibniz-Nambu-Poisson color algebra is said to
 be commutative. Moreover, if the trilinear map $[-, -, -]$ is $\varepsilon$-skew-symmetric for any pair of variables, then 
$(A, \cdot, [-, -, -], \varepsilon)$ is called a ternary Nambu-Poisson color algebra, see \cite{MA} for trivial grading. 
\end{definition}
\begin{example}
 If $(A, \cdot, [-, -, -], \varepsilon)$ is a non-commutative ternary Leibniz-Nambu-Poisson color  algebra, then
$(A, \cdot^{op}, [-, -, -], \varepsilon)$ is also a non-commutative ternary Leibniz-Nambu-Poisson color  algebra, where $\cdot^{op} : A\otimes A
\rightarrow A, x\otimes y \mapsto\varepsilon(x, y)y\cdot x$.
\end{example}
The following proposition asserts that the direct sum of two non-commutative ternary  Leibniz-Nambu-Poisson color  algebras are also a
non-commutative ternary Leibniz-Nambu-Poisson color  algebra with componentwise operation.
\begin{proposition}
 Let  $(A, \cdot, [-, -, -], \varepsilon)$ and $(A', \cdot', [-, -, -]', \varepsilon)$ be two non-commutative ternary Leibniz-Nambu-Poisson color 
 algebras. Then $A\oplus A'$ is a non-commutative ternary Leibniz-Nambu-Poisson color  algebra with respect to the operations :
\begin{eqnarray}
 (x\oplus x')\ast(y\oplus y')&:=&(x\cdot y)\oplus (x'\cdot'y')\nonumber\\
\{x\oplus x', y\oplus y', z\oplus z'\}&:=&[x, y, z]\oplus[x', y', z']'\nonumber,
\end{eqnarray}
for any $x, y, z\in \mathcal{H}(A)$ and $x', y', z'\in \mathcal{H}(A')$.
\end{proposition}
\begin{proof}
It is straightforward.
\end{proof}

The below theorem states that the tensor product of non-commutative ternary Leibniz-Nambu-Poisson color algebra by itself leads to
 non-commutative Leibniz-Poisson color algebra.
\begin{theorem}\label{llp}
 Let $(A, \cdot, [-, -, -], \varepsilon)$ be a non-commutative ternary Leibniz-Nambu-Poisson color  algebra. Then $A\otimes A$ is endowed with a 
non-commutative Leibniz-Poisson color algebra structure for the structure maps defined by
\begin{eqnarray}
 (x\otimes y)\cdot(x'\otimes y')&:=&\varepsilon(y, x')(x\cdot x')\otimes (y\cdot y')\nonumber\\
{[x\otimes y, x'\otimes y']}&:=&x\otimes[y, x', y']+\varepsilon(y, x'+y')[x, x', y']\otimes y.\nonumber
\end{eqnarray}
\end{theorem}
\begin{proof}
 It follows from a direct computation.
\end{proof}

The following results affirm that one can get non-commutative ternary Leibniz-Nambu-Poisson color algebras from non-commutative Leibniz-Poisson 
color algebras.
\begin{theorem}\label{ll}
 Let $(A, \cdot, [-, -], \varepsilon)$ be a non-commutative Leibniz-Poisson color algebra. Then \\ $(A, \cdot, [-, [-, -]], \varepsilon)$ is a non-commutative
 ternary Leibniz-Nambu-Poisson color algebra.
\end{theorem}
\begin{proof}
 It is clear that $(A, \cdot, \varepsilon)$ is an associative color algebra and $(A, [-, [-, -]])$ is a ternary Leibniz color algebra, thanks to
Theorem \ref{ll3}. We now need to prove the compatibility condition (\ref{cal3}). For any $x, y, z, t\in \mathcal{H}(A)$, we have
\begin{eqnarray}
 [x\cdot y, z, t]
&=&[x\cdot y, [z, t]]\nonumber\\
&=&x\cdot [y, [z, t]]+\varepsilon(y, z+t)[x, [z, t]]\cdot y\quad (\mbox{by}\quad (\ref{comp}))\nonumber\\
&=&x\cdot [y, z, t]+\varepsilon(y, z+t)[x, z, t]\cdot y.\nonumber
\end{eqnarray}
This ends the proof.
\end{proof}

\begin{corollary}
Let $(A, \cdot, \varepsilon)$ be an associative color algebra. Define the even bilinear and trilinear maps 
$[-, -] : A^{\otimes2}\rightarrow A$ and $[-, -, -] : A^{\otimes3}\rightarrow A$ respectively by
$$[x, y]:=x\cdot y-\varepsilon(x, y)y\cdot x \quad\mbox{and}\quad [x, y, z]:=[x, [y, z]]$$
Then $(A, \cdot, [-, -, -], \varepsilon)$ is a non-commutative ternary Leibniz-Nambu-Poisson color algebra.
\end{corollary}
\begin{corollary}
 Let $(A, \cdot, [-, -])$ be a non-commutative Leibniz-Poisson color algebra. Then $A\otimes A$ is also a non-commutative Leibniz-Poisson algebra
with respect to the operations 
\begin{eqnarray}
 (x\otimes y)\cdot(x'\otimes y')&:=&\varepsilon(y, x')(x\cdot x')\otimes (y\cdot y')\nonumber\\
{[x\otimes y, x'\otimes y']}&:=&x\otimes[y, [x', y']]+\varepsilon(y, x'+y')[x, [x', y']]\otimes y.\nonumber
\end{eqnarray}
\end{corollary}
\begin{proof}
 It follows from Theorem \ref{llp} and Theorem \ref{ll}. 
\end{proof}

It is proved in (\cite{JMC}, Lemma 2.8) that every non-commutative Leibniz-Poisson algebra gives rise to ternary Leibniz algebra. 
In the below theorem, we extend this result the non-commutative Leibniz-Nambu-Poisson color algebra case.
\begin{theorem}
 Let $(A, \cdot, [-, -], \varepsilon)$ be a non-commutative Leibniz-Poisson color algebra, then\\
$(A, \cdot, \{-, -, -\}, \varepsilon)$ is a non-commutative ternary Leibniz-Nambu-Poisson color algebra, with
$$\{x, y, z\}:=[x, y\cdot z]$$
for  any $x, y, z\in \mathcal{H}(A)$.
\end{theorem}
\begin{proof}
We only need to prove {\it ternary $\varepsilon$-Nambu identity} (\ref{lci}) and {\it right ternary Leibniz} rule  (\ref{cal3}). Then,
 for any $x, y, z, t, u\in \mathcal{H}(A)$, one has :
\begin{eqnarray}
 &&\qquad\{\{x, y, z\}, t, u\}-\{x, y, \{z, t, u\}\}-\varepsilon(z, t+u)\{x, \{y, t, u\}, z\}-\varepsilon(y+z, t+u)\{\{x, t, u\}, y, z\}=\nonumber\\
 &&={[[x, y\cdot z], t\cdot u]]}-[x, y\cdot [z, t\cdot u]]-\varepsilon(z, t+u)[x, [y, t\cdot u]\cdot z]
-\varepsilon(y+z, t+u)[[x, t\cdot u], y\cdot z]\nonumber\\
&&={[x, y\cdot z], t\cdot u]}-[x, y\cdot [z, t\cdot u]-\varepsilon(z, t+u) [y, t\cdot u]\cdot z]
-\varepsilon(y+z, t+u)[[x, t\cdot u], y\cdot z]\nonumber\\
&&={[[x, y\cdot z], t\cdot u]}-[x, [y\cdot z, t\cdot u]-\varepsilon(y+z, t+u)[[x, t\cdot u], y\cdot z].\nonumber
\end{eqnarray}
The last line vanishes thanks to the {\it right Leibniz} identity (\ref{cpa}).
Next,
\begin{eqnarray}
&&\qquad \{x\cdot y, z, t\}-x\cdot \{y, z, t\}-\varepsilon(y, z+t)\{x, z, t\}\cdot y=\nonumber\\
&&={[x\cdot y, z\cdot t]}-x\cdot [y, z\cdot t]+\varepsilon(y, z+t)[x, z\cdot t]\cdot y\nonumber,
\end{eqnarray}
which also vanishes by (\ref{cpa}).
\end{proof}

In what follows we introduce and give constructions of modules over non-commutative ternary Leibniz-Nambu-Poisson color algebras.

First, we need the below definition.
\begin{definition}

Let $(A,\cdot, \varepsilon)$ be an associative color algebra.
   An $A$-bimodule is  a $G$-graded vector space  $M$  together with two even
 bilinear maps $\prec M\otimes A\longrightarrow M$ and $\succ :A\otimes M\longrightarrow M$ called structure maps such that
\begin{eqnarray}
x\succ(y\succ m)&=&(x\cdot y)\succ m,\label{mha3}\\
 (m\succ x)\prec y&=&m\prec(x\cdot y),\\
x\succ (m\prec y)&=&(x\succ m)\prec y,\label{mha5}
\end{eqnarray}
for any $x, y\in \mathcal{H}(A)$ and $m\in \mathcal{H}(M)$.
\end{definition}

\begin{definition}
 A module over a non-commutative ternary Leibniz-Nambu-Poisson color algebra $(A, \cdot, [-, -, -], \varepsilon)$ is a bimodule
 $(M, \prec, \succ, \varepsilon)$ over the associative color algebra $(A, \cdot, \varepsilon)$ with three even trilinear applications
$$[-, -, -] : M\otimes L\otimes L\rightarrow M, \quad [-, -, -] : L\otimes M\otimes L\rightarrow M, \quad
[-, -, -] : L\otimes L\otimes M\rightarrow M$$
satisfying the following sets of axioms
\begin{eqnarray}
 [[m, x, y], z, t]]&=&[m, x, [y, z, t]]+\varepsilon(y, z+t)[m, [x, z, t], y]+\varepsilon(x+y, z+t)[[m, z, t], x, y],\label{lpc3a1}\\
{[[x, m, y], z, t]]}&=&[x, m, [y, z, t]]+\varepsilon(y, z+t)[x, [m, z, t], y]+\varepsilon(m+y, z+t)[[x, z, t], m, y],\\
{[[x, y, m], z, t]]}&=&[x, y, [m, z, t]]+\varepsilon(m, z+t)[x, [y, z, t], m]+\varepsilon(y+m, z+t)[[x, z, t], y, m],\\
{[[x, y, z], m, t]]}&=&[x, y, [z, m, t]]+\varepsilon(z, m+t)[x, [y, m, t], z]+\varepsilon(y+z, m+t)[[x, m, t], y, z],\\
{[[x, y, z], t, m]]}&=&[x, y, [z, t, m]]+\varepsilon(z, t+m)[x, [y, t, m], z]+\varepsilon(y+z, t+m)[[x, t, m], y, z], \label{lpc3a5}
\end{eqnarray}
and 
\begin{eqnarray}
 [m\cdot x, y, z]=m\cdot [x, y, m]+\varepsilon(x, y+z)[m, y, z]\cdot x,\\
{[x\cdot m, y, z]}=x\cdot [m, y, z]+\varepsilon(m, y+z)[x, y, z]\cdot m,\\
{[x\cdot y, m, z]}=x\cdot [y, m, z]+\varepsilon(y, m+z)[x, m, z]\cdot y,\\
{[x\cdot y, z, m]}=x\cdot [y, z, m]+\varepsilon(y, z+m)[x, z, m]\cdot y,\label{lpc3a7}
\end{eqnarray}
for any $x, y, z, t\in \mathcal{H}(A)$ and $m\in \mathcal{H}(M)$.
\end{definition}
\begin{remark}
Axioms (\ref{lpc3a1})-(\ref{lpc3a5}) mean that $M$ is a module over a ternary Leibniz color algebra.
\end{remark}

Next we have the following definition.
\begin{definition}
Let $(L, [-, -], \varepsilon)$ be a Leibniz color algebra and $M$ be a $G$-graded vector space. An $L$-module on $M$ consists of
even $\bf \mathbb{K}$-bilinear maps
$\ast : L\otimes M\rightarrow M$ and $\ast' : M\otimes L\rightarrow M$ such that for any $x, y\in \mathcal{H}(L), m\in\mathcal{H}(M)$,
\begin{eqnarray}
 [x, y]\ast m&=&x\ast(y\ast m)+\varepsilon(y, m)(x\ast m)\ast' y\label{lm11},\\
(m\ast' x)\ast' y&=&m\ast' [x, y]+\varepsilon(x, y)(m\ast' y)\ast' x,\label{lm22}\\
(x\ast m)\ast' y &=&x\ast(m \ast' y)+\varepsilon(m, y) [x, y]\ast m.\label{lm33}
\end{eqnarray}
\end{definition}
\begin{example}
 Any Leibniz color algebra or any Leibniz superalgebras \cite{CQZ} is a module over itself.
\end{example}
In the below proposition, we construct modules on ternary Leibniz color algebras from modules over Leibniz color algebras.
\begin{proposition}
 Let $(M, \ast, \ast', \varepsilon)$ be a module over a Leibniz color algebra $(L, [-, -], \varepsilon)$. Define
\begin{eqnarray}
 [x, y, m]:=x\ast (y\ast m),\quad [x, m, y]:=x\ast (m\ast' y), \quad [m, x, y]:=m\ast' [x,  y],
\end{eqnarray}
for all $x, y\in \mathcal{H}(L)$ and $m\in \mathcal{H}(M)$.\\
Then $M$ is a module over the ternary Leibniz color algebra associates to the Leibniz color algebra $L$ (as in theorem \ref{ll3}).
\end{proposition}
\begin{proof}
 It is proved by a straightforward computation.
\end{proof}
In order to have modules over non-commutative ternary Leibniz-Nambu-Poisson color algebras via morphism, we need the below definition.
\begin{definition}
 Let $(L, \cdot, [-, -, -], \varepsilon)$ and $(L', \cdot', [-, -, -]', \varepsilon)$ be two non-commutative ternary Leibniz-Nambu-Poisson
 color algebras. Let $\alpha : L\rightarrow L'$ be an even linear mapping such that, for any $x, y, z\in \mathcal{H}(L)$,
$$\alpha(x\cdot y)=\alpha(x)\cdot\alpha(y)\quad\mbox{and}\quad \alpha([x, y, z])=[\alpha(x), \alpha(y), \alpha(z)]'.$$
Then $\alpha$ is called a morphism of non-commutative ternary Leibniz-Nambu-Poisson color algebras.
\end{definition}
Then we have the following result.
\begin{theorem}
 Let $(L, \cdot, [-, -, -], \varepsilon)$ and $(L', \cdot', [-, -, -]', \varepsilon)$ be two ternary Leibniz-Nambu-Poisson color algebras and 
$\alpha : L\rightarrow L'$ be a morphism of non-commutative ternary Leibniz-Nambu-Poisson color algebras. Define 
\begin{eqnarray}
x\ast m &=&\alpha(x)\cdot' m, \quad m\ast x =m\cdot'\alpha(x), \quad\mbox{and}\quad\label{m1}\\
\{ m, x, y\}&=&[m, \alpha(x), \alpha(y)]',\quad
\{x, m, y\}=[\alpha(x), m, \alpha(y)]',\quad\mbox{and}\quad
\{x, y, m\}=[\alpha(x), \alpha(y), m]',\qquad\label{m2}
\end{eqnarray}
for any $x, y\in \mathcal{H}(L)$ and $m\in \mathcal{H}(L')$. Then, with these five maps, $L'$ is a module over $L$. 
\end{theorem}
\begin{proof}
For all $x, y, z\in \mathcal{H}(L)$, we have
\begin{eqnarray}
 \{\{x, y, z\}, t, m\}&
=&[[\alpha(x), \alpha(y), \alpha(z)]', \alpha(t), m]'\nonumber\\
&=&[\alpha(x), \alpha(y), [\alpha(z), \alpha(t), m]']'+\varepsilon(z, t+m)[\alpha(x), [\alpha(y), \alpha(t), m]', \alpha(z)]'\nonumber\\
&&+\varepsilon(y+z, t+m)[[\alpha(x), \alpha(t), m]', \alpha(y), \alpha(z)]'\nonumber\\
&=&\{x, y, \{z, t, m\}\}+\varepsilon(z, t+m)\{x, \{y, t, m\}, z\}+\varepsilon(y+z, t+m)\{\{x, t, m\}, y, z\}.\nonumber
\end{eqnarray}
The rest of the relations are proved similarly.
\end{proof}
\begin{remark}
 Any non-commutative ternary Leibniz-Nambu-Poisson color algebra is a module over itself. 
\end{remark}

\begin{corollary}
 Let $(L, \cdot, [-, -, -], \varepsilon)$ be a non-commutative ternary Leibniz-Nambu-Poisson color algebra and $\alpha : L\rightarrow L$ be an
 endomorphism of $L$. Then (\ref{m1}) and (\ref{m2}) define another module structure of $L$ over itself. 
\end{corollary}


{\bf Acknowlegments :} 
The author would like to thank Professor Alain Togb\'e of Purdue University for his material support.

\label{lastpage-01}
\end{document}